\documentclass[preprint,nonatbib]{elsarticle}
\makeatletter
\let\c@author\relax
\makeatother

\usepackage[official]{eurosym}
\usepackage{mathtools}
\usepackage{amsmath,amsthm}
\usepackage{amssymb}
\usepackage{mleftright}
\usepackage{dirtytalk}
\usepackage{xspace}
\usepackage{tikz}
\usetikzlibrary{backgrounds}
\usetikzlibrary{shapes.geometric, arrows,math,calc}
\usepackage[disable]{todonotes}
\usepackage[linesnumbered,ruled,vlined]{algorithm2e}
\usepackage[shortlabels]{enumitem}
\usepackage{algpseudocode}
\usetikzlibrary{patterns}
\usetikzlibrary{shapes.misc}
\usetikzlibrary{arrows.meta}
\usetikzlibrary{shapes.geometric}
\usepackage{suffix}
\usepackage{xstring}
\usepackage[nolist]{acronym}
\usepackage{float}
\usepackage{contour}

\usepackage{subcaption}
\usepackage[short]{optidef}
\usepackage[hidelinks]{hyperref}
\usepackage[capitalise]{cleveref}
\usepackage[
    backend=biber,
    natbib=true,
    url=false, 
    doi=true,
    eprint=false
]{biblatex}
\newtheorem{theorem}{Theorem}

\newtheorem{lemma}[theorem]{Lemma}
\theoremstyle{definition}
\newtheorem{definition}[theorem]{Definition}
\theoremstyle{remark}

\definecolor{RWTHblue}{RGB}{0,83,159}
\definecolor{RWTHblack}{RGB}{0,0,0}
\definecolor{RWTHwhite}{RGB}{255,255,255}
\definecolor{RWTHlightblue}{RGB}{142,186,226}
\definecolor{RWTHgrey}{RGB}{51,51,51}
\definecolor{RWTHlightgrey}{RGB}{204,204,204}
\definecolor{RWTHsuperlightgrey}{RGB}{247,247,247}
\definecolor{RWTHpetrol}{RGB}{0,97,101}
\definecolor{RWTHteal}{RGB}{0,152,161}
\definecolor{RWTHmaygreen}{RGB}{189,205,0}
\definecolor{RWTHgreen}{RGB}{87,171,39}
\definecolor{RWTHyellow}{RGB}{255,237,0}
\definecolor{rwthorange}{RGB}{246,168,0}
\definecolor{RWTHmagenta}{RGB}{227,0,102}
\definecolor{RWTHred}{RGB}{204,7,30}
\definecolor{RWTHbordeaux}{RGB}{161,16,53}
\definecolor{RWTHviolet}{RGB}{97,33,88}
\definecolor{RWTHpurple}{RGB}{122,111,172}

\definecolor{dan-blue}{RGB}{0,84,159}
\definecolor{dan-darkblue}{RGB}{25,25,200}
\definecolor{dan-red}{RGB}{204,7,30}
\definecolor{dan-darkred}{RGB}{100,2,10}
\definecolor{dan-yellow}{RGB}{234,231,159}
\definecolor{dan-green}{RGB}{87,171,39}
\definecolor{superRed}{RGB}{176,5,5}

\makeatletter
\newcommand{\setword}[2]{%
  \phantomsection
  #1\def\@currentlabel{\unexpanded{#1}}\label{#2}%
}
\makeatother

\begin{acronym}[fmisma] 
\acro{pra}[PRA]{patient-to-room assignment problem}
\acro{dpra}[dynamic PRA]{dynamic patient-to-room assignment problem}
\acro{pba}[PBA]{patient-to-bed assignment problem}
\acro{colbm}[Col-BM]{minimum color-degree perfect $b$-matching problem}
\acro{brm}[BRM]{blue-red matching problem}
\acro{bcm}[BCM]{bounded color matching problem}
\acro{lmm}[\textsc{LMM}]{labeled maximum matching problem}
\acro{dwbm}[\textsc{D-WBM}]{diverse weighted $b$-matching problem}
\acro{tbtsat}[\textsc{(3,B2)-Sat}]{(3,B2)-satisfiability problem}
\acro{tsat}[\textsc{3Sat}]{3-satisfiability problem}
\acro{1sat}[\textsc{1in3-3Sat}]{exact one-in-three 3-satisfiability problem}
\acro{ssum}[\textsc{SubsetSum}]{subset sum problem}
\acro{rbtp}[\textsc{Red-Blue TP}]{red-blue transportation problem}
\acro{rrpra}[RRPRA]{room-restricted patient-to-room assignment problem}
\acro{sat}[SAT]{\textsc{Satisfiability}}
\acro{cpp}[\textsc{PPP}]{partition permutation problem}
\acro{fmisma}[\textsc{FlexMISMA}]{flexible multithread interval scheduling problem with machine availabilities}
\acro{isma}[\textsc{ISMA}]{interval scheduling problem with machine availabilities}
\acro{maxflow}[\textsc{MaxFlow}]{maximum flow problem}
\acro{misma}[\textsc{MISMA}]{multithread interval scheduling problem with machine availabilities}
\acro{ptsp}[\textsc{PTSP}]{personnel task scheduling problem}
\acro{fisrc}[\textsc{FISRC}]{interval scheduling problem with a resource constraint}
\acro{ufp}[\textsc{UFP}]{unsplittable flow problem on a path}
\acro{sap}[\textsc{SAP}]{storage allocation problem}
\acro{vdp}[\textsc{DVDP}]{directed vertex-disjoint paths problem}
\acro{ip}[IP]{integer programming}
\acro{mip}[MIP]{mixed-integer programming}
\acroplural{ip}[IPs]{integer programming}
\acro{bip}[BIP]{binary integer program}
\acroplural{bip}[BIPs]{binary integer programs}
\acro{tc}[TC]{transfer condition}
\acro{srp}[SRP]{single-room problem}
\acro{rmp}[RMP]{roommate problem}
\acro{mwpm}[MWPM]{minimum weighted perfect matching problem}
\acro{npa}[NPA]{nurse-to-patient assignment}
\acro{los}[LOS]{length-of-stay}
\end{acronym}

\renewcommand{\P}{\mathcal{P}}
\newcommand{\nP}{P}
\newcommand{\priv}{\P^*}         
\newcommand{\potrmm}{\P^\mathrm{R}}
\newcommand{\R}{\mathcal{R}}
\newcommand{\nR}{R}
\newcommand{\T}{\mathcal{T}}
\newcommand{\nT}{T}

\newcommand{\fp}{\P^{\mathrm{f}}}  
\newcommand{\nfp}{F}  
\renewcommand{\mp}{\P^{\mathrm{m}}}  
\newcommand{\nmp}{M}  
\newcommand{\rpold}{\mathcal{F}}          

\newcommand{\F}{\mathcal{F}}          
\newcommand{\powerset}{\mathbb{P}}
\newcommand{\Prt}{\P^{\ass}}
\newcommand{\M}{\mathcal{M}}
\renewcommand{\S}{\mathcal{S}}

\newcommand{\arr}{a}    
\newcommand{\dis}{d}    
\newcommand{\age}{e}    
\newcommand{\pref}{w}

\newcommand{\ftrans}{f^\mathrm{trans}}    

\newcommand{\fpriv}{f^\mathrm{priv}}      
\newcommand{\fpref}{f^\mathrm{\pref}}      
\newcommand{\wmin}{\pref^{\mathrm{min}}}

\newcommand{\rc}{c}     

\newcommand{\ass}{z}    

\newcommand{\smax}{s^{\max}}
\newcommand{\los}{\mathrm{los}}
\newcommand{\cl}{\ell}
\newcommand{\inst}{\mathcal{I}}

\newcommand{\bigabs}[1]{\big|#1\big|}

\renewcommand{\j}{\ensuremath{j}\xspace}

\makeatletter
\newcommand{\biggTwo}{\bBigg@{3}}
\newcommand{\biggFive}{\bBigg@{7.5}}
\makeatother
\DeclarePairedDelimiter\ceil{\lceil}{\rceil}
\DeclarePairedDelimiter\floor{\lfloor}{\rfloor}

\newcommand{\N}{\mathbb{N}}

\renewcommand{\O}{\mathcal{O}}

\newcommand{\abs}[1]{|#1|}
\newcommand{\pluseq}{\mathrel{{+}{=}}}

\definecolor{rwth-blue}{RGB}{0,83,159}
\definecolor{rwth-red}{RGB}{255,20,20}
\definecolor{rwth-lred}{RGB}{255,60,60}
\definecolor{rwth-green}{RGB}{87,171,39}
\definecolor{rwthorange}{RGB}{246,168,0}
\definecolor{RWTHorange}{RGB}{246,168,0}
\definecolor{rwthmaigruen}{RGB}{189,205,0}
\definecolor{rwthtuerkis}{RGB}{0,152,161}
\definecolor{RWTHred}{RGB}{204,7,30}
\definecolor{RWTHmaygreen}{RGB}{189,205,0}
\definecolor{combi-green}{RGB}{180,210,50}
\definecolor{combi-orange}{RGB}{250,170,30}
\newcommand{\fcolor}{rwthorange}
\newcommand{\mcolor}{rwthtuerkis}
\newcommand{\colorOne}{rwth-lred}
\newcommand{\colorTwo}{rwth-blue}

\tikzset{square/.style={regular polygon, regular polygon sides=4}}
\tikzset{patient/.style={square,fill=#1,draw=#1, minimum width=0.95cm}}
\tikzset{stay/.style={#1,line width=4.5pt}}
\tikzset{room/.style={square,fill=none,draw=black,minimum width=0.90cm}}
\tikzset{note/.style={draw=none,fill=none}}
\tikzset{%
	round/.style={circle,draw=black, fill=none, inner sep = 2.0pt},
	square/.style={regular polygon,regular polygon sides=4,draw=black, fill=none, inner sep = 2.0pt},
	Text/.style={draw=none,fill=none},
	kanten1/.style={color=\colorOne, ultra thick},
	kanten2/.style={color=\colorTwo, ultra thick}
}

\newcommand{\patient}[4]{
    \foreach \i [evaluate=\i as \m using \i-1] in {1,...,#2}{
        \node[patient=#1](#3\i) at ($#4+(\m*1,0)$) {}; 
    }   
    \ifnum#2>1
        \foreach \i [evaluate=\j as\m using \i-1] in {2,...,#2}{
            \draw[stay=#1] (#3\m) -- (#3\i);
        }   
    \fi 
}

\begin{document}
\begin{frontmatter}
\title{Combinatorial and Computational Insights about Patient-to-room Assignment under Consideration of Roommate Compatibility}
\author[aac]{Tabea Brandt\corref{cor1}\fnref{fn1}}
\ead{brandt@combi.rwth-aachen.de}
\author[aac]{Christina B\"using\corref{cor1}\fnref{fn1}} 
\ead{buesing@combi.rwth-aachen.de}
\author[aac]{Felix Engelhardt\corref{cor1}\fnref{fn1}}
\ead{engelhardt@combi.rwth-aachen.de} 

\address[aac]{Chair of Combinatorial Optimization, RWTH Aachen University, \\
Ahornstraße~55, 52074 Aachen, Germany}
\fntext[1]{
This work was supported
by the German research council (DFG) Research Training Group 2236 UnRAVeL.
Declarations of interest: none
}
\cortext[cor1]{Corresponding author}

\journal{arXiv.org}
%
\begin{abstract}
During a hospital stay, a roommate can significantly influence a patient's overall experience both positivly and negatively.
Therefore, hospital staff tries to assign patients together to a room that are likely to be compatible.
However,
there are more conditions and objectives to be respected by the
patient-to-room assignment (PRA), e.g.,
ensuring gender separated rooms and avoiding transfers.
In this paper, we review the literature for reasons why roommate compatibility is important as well as for criteria that can help to increase the probability that two patients are suitable roommates.
We further present combinatorial insights about computing patient-to-room assignments with optimal overall roommate compatibility.
We then compare different IP-formulations for PRA as well as the influence of different scoring functions for patient compatibility on the runtime of PRA  integer programming (IP) optimisation.
Using these results and real-world data, we conclude this paper by developing and evaluating a fast IP-based solution approach for the dynamic PRA.
\end{abstract}
\end{frontmatter}

\section{Introduction}
In hospitals, patients tend to spent the majority of their time not in surgery, but in their bed in their designated hospital room. 
The \ac{pra} problem is the operational problem of allocating patients to such hospital rooms while respecting room capacities, hospital policies and medical constraints. 
The \ac{pra} was first formulated by Demeester et al.~\cite{Demeester_2007}. Their definition of \ac{pra} already contains one type of user preference, i.e., patient preferences regarding double/single room choice~\cite{Demeester_2007}. This important special case together with the role of patient transfers is covered in detail by Büsing et al.~\cite{PRAComBIP}. 

In this work, we systematically consider different patients' preferences regarding suitable roommates as optimization objectives. 
However, what does constitute a suitable roommate is by no means an easy question.
Andersen et al. performed a meta-synthesis of qualitative studies on the subject; they identify three categories of how patients experience their roommate: as an enforced companion, as an expert on illness and hospital life, and/or as a care provider~\cite{Andersen_2012}. Such a companion can be the source of grief, as found in Adams et al.'s work on noise and sleep quality~\cite{Adams_2024}, or a roommate can be source of company (and loss of privacy), as found by Roos et al.~\cite{Roos_2009}.

Patient roommate preferences also link in with the ongoing scientific debate about the advantages and disadvantages of single vs. multiple bed room ward designs (e.g., see \cite{vandeGlind_2007,Persson_2012,Persson_2015}), which is not the focus of this work. However, we would like note that many factors considered advantages or disadvantages of multi bed rooms, especially concerning social aspects, are highly dependent on the specific roommate chosen.

It is important that roommate preferences are not just preferences, but translate into healthcare outcomes. A recent study by Sehgal, who analyses a large hospital dataset with respect to the impact of socio-demographic factors and clinical condition a patient's roommates on clinical outcomes, found that these factors influence patients length of stay~\cite{Sehgal_2023}.
An older but similar result is by Kulik et al., who showed that having a pre- or post-operative roommate can have an impact on patients' healthcare outcomes~\cite{Kulik1993StressAA,Kulik1996SocialCA}. In talks with practitioners, we found healthcare issues with "patient conflicts", e.g., having multiple patients with asthma in the same room are to be avoided.

Additionally, there is also an economic case to be made for considering patient preferences in \ac{pra}. Luther et al. look at patient roommates as a type customer-to-customer interaction. They find that this interaction affects overall patient satisfaction with a hospital~\cite{Luther_2016}, which is validated with a follow-up study by Hantel and Benkenstein~\cite{Hantel_2020}. 
Underlining this, in a quantitative data analysis, Young and Chen find that "Hospitality experiences create a halo effect of patient goodwill, while medical excellence and patient safety do not."~\cite{Young_2020}. This fits hospital practice, where practitioners tend to group patients together by age, types and/or severity of illnesses to improve patient satisfaction~\cite{Overflow}.
Thus, assigning patient roommates that are in line with their preferences can increase overall patient satisfaction at little additional cost. This is important, as patient satisfaction is not only an indicator for quality of care~\cite{clearly_1998}, but also closely linked to patient loyalty~\cite{Kessler_2011,Liu_2021}.

The remainder of the paper is structured as follows:
We first give a short overview over related literature in \Cref{sec:literature}, and formally define \ac{pra} in \Cref{sec:def}.
In \cref{sec:pref:complexity}, we consider the complexity of finding a feasible patient-to-room assignment with optimal roommate suitability and show that such a solution can be found in polynomial time.
Then, in \Cref{sec:pref:patient-fit}, we categorize existing research on patient preferences from a mathematical modelling perspective, and derive both general classes and specific examples of preference functions.
In \cref{sec:pref:lp}, we propose and computationally compare different \ac{ip} formulations for \ac{pra} under consideration of roommate compatibility.
Then, in \Cref{sec:pref:dyn} we consider different integer programming formulations for modelling preference objectives. These formulations are then compared in computational study.
Finally, we end with a short conclusion in \cref{sec:conclusion} focusing on implications for practical modelling and further research.

\section{Literature}\label{sec:literature}

There exists various literature regarding hospital bed management.
For example, it was suggested in~\cite{Utilisation} that health-care spending can be reduced by planning the utilisation of hospital resources efficiently.
In~\cite{simu1,simu2,simu3}, simulation models were developed to help evaluate such planning strategies.
Furthermore, a decision support system for bed-assignments has been proposed in~\cite{1995} which aims to mimic decisions that are performed by professional hospital staff.

The specific problem that we consider and expand on in this paper was introduced and modelled by~\citet{Demeester_2010} in 2010: the operational task of assigning patients to hospital-rooms subject to hospital policies, where both the wishes of patients and medical nursing needs are considered in order to determine how well each room is suited for a patient.
This problem is known as the patient-to-room assignment problem.

In the original definition,\citet{Demeester_2010} assumed that both the admission and departure date of each patient is known in advanced.
Ceschia and Schaerf expanded on this task in \cite{LocalSearch} by considering emergency patients as well, for which these dates are not known until the patients are admitted.
\todo[inline]{Add remarks on which kind of patient preferences Demeester et al. consider in their objective function}

Since 2010, various different approaches to solve \ac{pra} have been proposed, such as a tabu search algorithm \cite{Demeester_2010}, a metaheuristic \cite{OfflinePat}, a simulated annealing algorithm \cite{LocalSearch} and a hybrid simulated annealing algorithm \cite{HybridAnn}.
Furthermore, many variations of \ac{pra} have been proposed: in \cite{LocalSearch}, the problem has been extended to a dynamic context to which the authors adapted their solution technique that was proposed to the static problem.
Vancroonenburg et al. proposed an integer linear program to solve a similar dynamic extension of the \ac{pra} in \cite{DecSupport}, by considering uncertainty of patients' length of stay as well as emergency patients whose arrival date is only known at the time when they are admitted to the hospital.
An issue that occurs for \ac{pra} in a dynamic context is that the quality of the solution for the whole time horizon (long-term objective) is often bad when repeatedly optimising the assignment of newly registered patients (short-term objective).
In \cite{ShortLongTerm}, Yi-Hang Zhu et al. studied the compatibility of such short-term and long-term objectives. Schäfer et al. proposed a decision support model in \cite{Overflow} that anticipates trade-off between the objective value and handling overflow situations created both by arrivals of emergency patients and the uncertainty of the patients' length of stay, based on historical probability distributions.
Schmidt et al. examined efficiency of methods to solve \ac{pra} with elective patients, where uncertainty of each patient's length of stay is considered in \cite{Schmidt}.
This was done by modelling each patient's length of stay by a log-normally distributed random variable for each patient.
Representing this problem as a binary integer program, three heuristics, an exact approach and a mixed integer program solver were introduced and compared. 
In \cite{OfflinePat}, a detailed comparison of variations of the problem that is considered in different papers is given.
While patients' preferences for room-equipment or the layout of rooms have been considered for the patient-to-room assignment already, e.g., in \cite{Demeester_2010,DecSupport,Overflow,LocalSearch,OfflinePat}, preferences regarding room-mates have not yet been modelled for \ac{pra}.
In real-life situations, such wishes are sometimes considered by hospital staff \cite{Roomies}.

Further and independently of PRA, there are various papers that examine which preferences patients have regarding room-layout and equipment, such as \cite{RoomEq2,RoomEq3,RoomEq4}. 
For example, evidence suggests that in general, patients experience a more pleasant stay if they are assigned into single- or double-rooms, according to \cite{Zufriedenheit}. However, only little research about preferences regarding room-mates can be found. 
In \cite{Conversations}, patient-preferences of US-veterans were analysed. It was observed that the most positive aspect for patients in shared rooms is conversations with room-mates. Therefore, it is desirable for such patients to be assigned room-mates that share similar interests. While there are many factors that determine how well patients get along, these are hard to establish beforehand for hospital staff. However, one factor that can easily be taken into account is the age of each patient. In \cite{AgeDifference}, it was observed that a large age difference between room-mates sharing a hospital room is often undesirable for patients. 

\section{Problem Definition}\label{sec:def}
\todo[inline]{introduce maximization of single-room requests}
Mathematically, we consider \ac{pra}, over a finite planning horizon $\T:=\left\{1,\hdots,\nT\right\}$ and a set of rooms $\R=\left\{1,\hdots,\nR\right\}$.
Each room $r\in\R$ has a capacity $\rc_r\in\N$ which is equal to its number of beds located in $r$.
Further, we have a set $\P=\left\{1,\hdots,\nP\right\}$ of patients.
The set $\P$ is the disjoint union of female patients $\fp$ and male patients $\mp$, i.e., $\P=\fp\cup \mp$.
Each patient $p\in \P$ has an arrival period $\arr_p\in \T$ as well as a discharge period $\dis_p\in \T$.
We assume that a patient does no longer need a bed in this last time period and define the set of time periods for which patient $p$ has to be assigned to a bed as
\[\T(p):=\{t\in\T \mid \arr_p\leq t < \dis_p\}.\]
Based on the arrival and discharge dates of each patient, we denote the set of patients that need to be assigned to a room during time period $t\in\T$ as
\[\P(t):=\{p\in\P \mid t\in \T(p)\}.\]
Analogously, we define for any subset $S\subseteq\P$ of patients the respective subset of patients who need a room in time period $t\in\T$ as $S(t):=S\cap\P(t)$.
We suppose that both the arrival date and the discharge date is fixed in advance.

We further assume that we have a black box function $\pref$ which efficiently computes
a natural value that denotes the roommate fit.
We define $\pref$ over the set of all subsets of patients, i.e., the power set $\powerset(\P)$ of $\P$ which is defined as
\[ \powerset(P):=\{S\subseteq\P\}.\]
We then define the preference function $\pref$ as
\[\pref:\powerset(P)\rightarrow\N_0,S\mapsto \pref(S),\]
where $\pref(S)=0$ means the patients in $S$ fit perfectly together as roommates.
In general, specifying $\pref(S)$ is a highly complex task in real life.
We elaborate possible criteria that can be taken into account in \cref{sec:pref:patient-fit} and discuss respective properties of $\pref$.
Remark that it is also possible to define $\pref$ such that a higher value corresponds to a better fit.
For the structural analysis in this paper, we prefer our definition.

The task in \ac{pra} is to allocate every patient $p\in\P$ to a hospital room $r\in\R$ for each $t\in\T(p)$ so that each room accommodates only either male or female patients and the rooms' capacities are respected while maximizing the overall roommate fit and minimizing the total number of transfers.
That means we search an assignment
\[\ass:\{(p,t)\in\T\times\R\mid t\in\T(p)\}\rightarrow\R\]
subject to:
\begin{enumerate}
    \item Each patient $p\in \P$ is assigned to exactly one room $r\in\R$ in each time period $t\in\T(p)$, i.e., $\ass$ is well defined.
    \item The rooms' capacities are respected in every time period, i.e.,
        \[\abs{\Prt(r,t)}:= \abs{\{p\in\P\mid\ass(p,t)=r\}}\leq c_r\quad\forall r\in\R,~t\in\T.\]
    \item Female and male patients never share a room, i.e.,
        \[\ass(\fp,t)\cap\ass(\mp,t)=\emptyset\quad\forall t\in\T.\]
    \item The overall roommate fit is as good as possible, i.e., we minimize
        \[\fpref(z):=\sum_{r\in\R}\sum_{t\in\T} \pref(\Prt(r,t)).\]
    \item We need as few transfers of patients between rooms as possible, i.e., we minimize
        \[\ftrans(\ass):=\sum_{p\in\P}\left|\left\{t\in\T(p)\setminus\{\arr_p\}: \ass(p,t)\neq\ass(p,t-1)\right\}\right|\]
\end{enumerate}
\section{Complexity}
\label{sec:pref:complexity}

In this section, we first define the task of determining an optimal roommate choice for a specific day as a combinatorial optimization problem: the \acf{rmp}.
We then examine the underlying combinatorial structure and expose connections to other known combinatorial problems.
Remark that we can compute the best roommate-fit value independently for every single time period (since we allow arbitrary many transfers).


Therefore, we consider in this section only one arbitrary but fixed time period $t\in\T$
and abbreviate the set of female patients who are in hospital in time period $t$ with
\[\F:=\fp(t),\]
and respectively the set of male patients needing a bed in time period $t$ with
\[\M:=\mp(t).\]
We further denote their cardinality with $\nfp:=\abs{\F}$ and $\nmp:=\abs{\M}$, and with
\[\nR_\rc:=\abs{\{r\in\R\mid \rc_r=\rc\}}\]
the number of rooms with a specific capacity $\rc\in\N$.

We define the problem of finding an optimal partition of patients into roommates in time period $t$ as follows.
\begin{definition}[\ac{rmp}]
Let a set of female patients $\F$, a set of male patients $\M$, a set of rooms $\R$ with capacities $\rc_r\in\N$ for $r\in\R$, and a scoring function $\pref:\powerset(\P)\rightarrow\N_0$ be given.
In the \acf{rmp}, the task is then to find a collection $\S$ of subsets of patients that represent a feasible patient-to-room assignment with optimal patient fit, i.e.,
\[\S:= \{S_r\subseteq \P \mid r\in \R\},\]
so that
\begin{enumerate}[i)]
    \item every patient is assigned to at most one room, i.e., $S_r\cap S_u=\emptyset$ f.a. $r\in\R$,\label{pref:itm:exone}
    \item every patient is assigned to a room, i.e., $\bigcup_{r\in\R} S_r = \F\cup\M$,\label{pref:itm:everypat}
    \item rooms are separated by sex, i.e., $S_r\cap \F=\emptyset$ or $S_r\cap\M=\emptyset$ f.a. $r\in\R$,\label{pref:itm:sex}
    \item the room capacities are respected, i.e., $\abs{S_r}\leq \rc_r$ f.a. $r\in\R$,\label{pref:itm:cap}
    \item the overall patient fit is optimal, i.e., $\pref(\S):=\sum_{r\in\R} \pref(S_r)$ is minimal.\label{pref:itm:opt}
\end{enumerate}
\end{definition}

Remark that we can check the feasibility of a given \ac{rmp} instance as described in \cite{brandt2024privates}.
We therefore concentrate in the remaining part of this section on instances whose feasibility we can check in polynomial time and especially on the case of room capacities $\rc_r\in\{1,2\}$.
Since \ac{rmp} is trivial for instances with only single rooms, we further assume that at least one double exists, i.e., $\nR_2\geq 1$.
Then, we can find an optimal solution to a feasible \ac{rmp} instance by solving a \acf{mwpm}:
Let a feasible \ac{rmp} instance $\inst$ be given, we then construct a weighted and undirected graph $G_\inst=(V,E,\cl)$ as follows.
We construct a vertex for every patient and $k:=2\nR - \nfp - \nmp$ auxiliary vertices representing single rooms and beds that stay empty, i.e.,
\[ V:=\F\cup\M\cup \underbrace{\{v_1,\ldots,v_k\}}_{=:X}.\]
Then, edges of a perfect matching in $G_\inst$ can connect
\begin{enumerate}[i)]
    \item two patient-vertices, representing them as roommates;
    \item a patient vertex and an auxiliary vertex, representing that the patient is alone in a room;
    \item two auxiliary vertices, representing an empty room.
\end{enumerate}

However, we have to ensure that the number of matching edges connecting two patient vertices is less or equal to the number of double rooms.
We therefore compute the number $\alpha$ of single rooms that have to be used as
\[\alpha := \max\left\{\nfp+\nmp-2\nR_2,\ 0\right\}.\]
We then use the first $\alpha$ auxiliary vertices to ensure that there are at least $\alpha$ edges connecting patient vertices with auxiliary vertices in every perfect matching.
For easier lucidity, we denote the set of these vertices by $X_1:=\{v_1,\ldots,v_\alpha\}$ and the remaining auxiliary vertices as $X_2:=\{v_{\alpha+1},\ldots,v_k\}$.
Remark that $X_1=\emptyset$ if $\alpha=0$.

We construct three types of edges, cf.~\cref{fig:pref:ginst}:
between all vertices that represent potential roommates, between all patient vertices  and auxiliary vertices, and between all vertices in $X_2$, i.e.,
\begin{align*}
E:= ~&\{pq \mid p,q\in\F \text{ or } p,q\in\M\}\\
\cup~&\{px \mid p\in \F\cup\M,\ x\in X\}\\
\cup~&\{xy\mid x,y\in X_2\}.
\end{align*}

\begin{figure}[tbh]
\begin{center}
\resizebox{0.3\textheight}{!}{
\begin{tikzpicture}
\contourlength{2pt}
\node[circle, draw=\fcolor, inner sep=0.5cm,pattern=crosshatch, pattern color=\fcolor](F) at (0,0) {\contour{white}{\textcolor{black}{$K_F$}}};
\node[circle, draw=\mcolor, inner sep=0.5cm,pattern=crosshatch, pattern color=\mcolor](M) at (4,0) {\contour{white}{\textcolor{black}{$K_M$}}};
\node[ellipse, draw=black, inner sep=0.5cm,minimum width=3cm,pattern=crosshatch, pattern color=black](X2) at (3,4) {};
\node[ellipse, draw=black, inner sep=0.75cm, minimum width=6cm](X) at (2,4) {};
\node[circle, draw=black, fill=black, inner sep=1.0pt, label=above:{\contour{white}{$v_1$}}](v1) at (-0.5,4){};
\node[](s) at (0.25,4){$\ldots$};
\node[](tX) at (4,5){$X$};
\node[](tX2) at (1.8,4.6){\contour{white}{$X_2$}};
\node[circle, draw=black, fill=black, inner sep=1.0pt, label=above:{\contour{white}{$v_\alpha$}}](v2) at (1,4){};
\draw (M) -- (X2) -- (F);
\draw (M.north east) -- (X2) -- (F.east);
\draw (M.north) -- (X2) -- (F.north);
\draw (M) -- (v1) -- (F);
\draw (M.west) -- (v1) -- (F.west);
\draw (M.north) -- (v1) -- (F.north);
\draw (M) -- (v2) -- (F);
\draw (M.west) -- (v2) -- (F.east);
\draw (M.north) -- (v2) -- (F.north);
\end{tikzpicture}
}
\end{center}
\caption{Graph $G_\inst$}
\label{fig:pref:ginst}
\end{figure}
We define the cost function $\cl$ using the patient-fit function $\pref$:
\[ \cl: E \rightarrow \N_0, uv \mapsto \begin{cases}
    \pref(\{u,v\}), &\text{if } u,v\in \F\cup\M\\
    \pref(\{v\}), &\text{if } v\in \F\cup\M,\ u\in X\\
    \pref(\emptyset), & \text{otherwise}.
    \end{cases}\]

We prove that an \ac{rmp} instance $\inst$ can be solved by computing a minimum-weighted perfect matching in $G_\inst$ in two steps.
First, we prove that $G_\inst$ always contains a perfect matching (if $\inst$ is feasible).
Second, we show how a perfect matching in $G_\inst$ corresponds to a solution of \ac{rmp}.
\begin{lemma}
Let $\inst$ be a feasible \ac{rmp} instance with room capacities $\rc_r\in\{1,2\}$ and at least one double room. The graph $G_\inst$ then contains a perfect matching.
\end{lemma}
\begin{proof}
Let $k:=2\nR - \nfp - \nmp$. If $k=0$, then $G_\inst$ consists of the two complete subgraphs $G_\inst[\F]$ and $G_\inst[\M]$.
Since $\inst$ is feasible, we have $\rc_r=2$ f.a. $r\in\R$ and both $\nfp$ and $\nmp$ are even which implies the existence of a perfect matching.

We prove that $G_\inst$ also contains a perfect matching if $k\geq 1$ by using Tutte's $1$-factor theorem.
Therefore, let $U\subsetneq V$.
If $U=\emptyset$, then we have no odd components in $G_\inst-\emptyset=G_\inst$ since $G_\inst$ is connected with $\abs{V}=2\nR$.

So, let $\abs{U}\geq 1$.
By construction of $G_\inst$, there are only two possibilities to choose $U$ so that $G_\inst -U$ is no longer connected:
\begin{enumerate}
    \item $X\subseteq U$, then $G_\inst -U$ is a subgraph of $G_\inst[\F\cup\M]$ which consists of at most two connected components.
        Therefore, it remains to show here that $\abs{U}=1$ implies that at most one of $G_\inst[\F]$ or $G_\inst[\M]$ odd, i.e., either $\nfp$ or $\nmp$ is even.
Since the original \ac{rmp} instance is feasible, we have
        \[1=\abs{U}=\abs{X}=k=2\nR - \nfp - \nmp.\]
        Which implies directly that one of $\nfp$ and $\nmp$ is even and the other one is odd.
        Hence, $G_\inst -U$ has at most $\abs{U}$ odd components.
    \item $\F\cup\M\subseteq U$, then $G_\inst -U$ is a subgraph of $G_\inst[X]$ which consists of at most $\alpha+1$ connected components.
        If $\alpha=0$, then $G_\inst-U$ is connected as $G_\inst[X_2]$ is a complete graph.
        So, let $\alpha\geq 1$. Then $G_\inst-U$ has at most $\alpha+1$ connected components.
        Since $\nR_2\geq 1$, we have
        \[\alpha+1=\nfp+\nmp-2\nR_2 +1\leq\nfp+\nmp -1 < \nfp+\nmp \leq \abs{U}.\]
        Hence, $G_\inst -U$ has at most $\abs{U}$ odd components.
\end{enumerate}

Overall, $G_\inst-U$ contains at most $\abs{U}$ odd components for every $U\subseteq V$.
Therefore, $G_\inst$ contains a perfect matching according to Tutte's $1$-factor theorem.
\end{proof}

\begin{lemma}
Let $\inst$ be a feasible \ac{rmp} instance.
A minimum-weight perfect matching in $G_\inst$ corresponds to an optimal solution of $\inst$.
\end{lemma}
\begin{proof}
Let $H\subseteq E$ be a minimum-weight perfect matching in $G_\inst$.
We define the collection $\S$ as follows,
\begin{align*}
\S:=&\{\{p,q\}\mid pq\in H,\ p,q\in\F\cup\M\}\\
\cup&\{\{p\}\mid px\in H,\ p\in\F\cup\M,\ x\in X\}\\
\cup&\{\emptyset \mid xy\in H, x,y\in X\}
\end{align*}
Condition \ref{pref:itm:exone} is satisfied as $H$ is a matching,
condition \ref{pref:itm:everypat} is satisfied as the matching $H$ is perfect, and
condition \ref{pref:itm:sex} is satisfied by construction of the graph $G_\inst$.
We show that condition \ref{pref:itm:cap} is also satisfied by proving that there are not more than $\nR_2$ elements with cardinality $2$ in $\S$.
Assume to the contrary that there are $\beta > \nR_2$ matching edges between patient vertices in $H$.
From $\beta > \nR_2$ it follows directly that there must be enough patients to feasibly fill $\beta$ double rooms, i.e.,
\[\floor*{\frac{\nfp}{2}}+\floor*{\frac{\nmp}{2}}\geq \beta > \nR_2.\]
Thus, we also have $\nfp+\nmp > 2\nR_2$ which implies $\alpha \geq 1$.
The total number of patient vertices that are incident to an edge in $H$ is therefore
\[2\beta + \alpha > 2\nR_2 + \nfp+\nmp-2\nR_2 = \nfp+\nmp.\]
This is a contradiction as we only have $\nfp+\nmp$ patient vertices in $G_\inst$.
Hence, condition \ref{pref:itm:cap} is satisfied and $\S$ represents a feasible solution for $\inst$.

It remains to show that $\S$ represents an optimal solution.
Assume to the contrary that there exists a solution $\S'$ for $\inst$ with $\pref(\S') < \pref(\S)$.
Let \[\S'_2:=\{S\in \S'\mid \abs{S}=2\}\]
be the set of all roommate pairs and
\[\S'_1:=\bigcup_{S\in \S': \abs{S}=1} S = \{y_1,\ldots,y_u\}\]
be the set of all patients who are alone in a room.
As $\S'$ is a feasible solution for $\inst$, we have $\abs{\S'_2}\leq \nR_2$ and $\abs{\S'_1}+\abs{\S'_2}\leq \nR_1+\nR_2$.
Then, it follows directly that $\alpha\leq \abs{\S'_1}$.
Moreover, it follows with $2\abs{\S'_2}+\abs{\S'_1}=\nfp+\nmp$ that
\begin{align*}
k-\abs{\S'_1} &= 2\nR_2+2\nR_1-\nfp-\nmp-\abs{\S'_1}\\
    &= 2\nR_2+2\nR_1-2\abs{\S'_2}-2\abs{\S'_1}\\
    &\geq 0
\end{align*}
and even.
Hence, we can define an edge set $H'\subseteq E$ that represents solution $\S'$ in $G_\inst$ as
\begin{align*}
H':=&\{pq\mid \{p,q\}\in \S'_2\}\\
\cup&\{y_i v_i \mid 1\leq i \leq u\}\\
\cup&\{v_{u+1}v_{u+2},v_{u+3}v_{u+4},\ldots,v_{k-1}v_k\}
\end{align*}
which is a perfect matching in $G_\inst$ by construction.
Since \[\cl(H')=\pref(\S')<\pref(\S)=\cl(H),\] this is a contradiction to $H$ being a minimum-weight perfect matching in $G_\inst$.
\end{proof}

With the graph $G_\inst$, we can thus solve the \ac{rmp} to optimality using known algorithms for the \acl{mwpm}.
\begin{theorem}
\label{thm:solveRMP}
\ac{rmp} with capacities $\rc_r\in\{1,2\}$ can be solved to optimality in $\O(n^3)$.
\end{theorem}
\begin{proof}
The graph $G_\inst$ can be constructed in $\O(n^2)$ with $n=2\nR$ and \citet{KorteVygen} report an algorithm for \acl{mwpm} with runtime $\O(n^3)$.
\end{proof}

Using \cref{thm:solveRMP}, we can compute the optimal roommate-fit score $\wmin_t$ for each time period $t\in\T$.
Using the exact computation of $\wmin_t$, we know that their sum over all time periods $t\in\T$ is a tight lower bound on the total objective value for $\fpref$, i.e.,
\begin{equation} \label{eq:wmin}
\fpref\geq \wmin :=\sum_{t \in \T} \wmin_t.
\end{equation}
This bound can always be achieved as long as arbitrary many transfers may be used.

We use these combinatorial insights in \cref{sec:pref:lp} to compute the optimal value of $\fpref$.
This allows us, on the one hand, to assess the quality of a found solution w.r.t. the roommate suitability, and
on the other hand, we can also integrate it directly into our dynamic solution approach.
However, in order to integrate roommate suitability into our solution approach for \ac{pra}, we need a calculable scoring function.
In the next section, \cref{sec:pref:patient-fit}, we first give an overview about criteria for roommate suitability that are reported in literature.
Second, we present the five scoring functions that we use to evaluate how the choice of scoring function influences the runtime and solution quality in \cref{sec:pref:lp}.

\section{Measuring the roommate fit}
\label{sec:pref:patient-fit}
Assessing how roommates influence a patient's satisfaction and health outcomes in hospital is an ongoing discussion in literature.
For surgical wards, already in 1983 Kulik et al. reported  that sharing a room with a post-operative patient statistically decreases a patient's anxiety of their own surgery~\cite{Kulik_1993}.
Later, they also found that having a roommate statistically decreases a patient's post-operative recovery time~\cite{Kulik_2000}.

Another positive influence is reported by recent studies on oncology wards.
There, fellow roommates with similar diagnosis, severity and age can be a great source of information and hope~\cite{Isaksen_2000}.
However, it is also reported that the loss of privacy can lead to patients withholding relevant information or preventing them from asking questions~\cite{Andersen_2014}.
Positive effects are foremost reported if patients can relate to each other well.

Being able to relate and communicate not only increases patients' well-being but can also measurably reduce patients' length-of-stay.
Schäfer et al. report that the age difference is generally a good indicator for predicting how well patients get along~\cite{Schaefer2019}.
However, \citet{Hantel_2019} report that roommates' personality and behaviour is more significant.
\citet{Sehgal_2023} study the impact of numerous social attributes on the length-of-stay to determine which of them can be used as indicator for (statistically) good roommate pairs.

Looking at the criteria mentioned by \citet{Sehgal_2023}, they can be described by scoring functions of one of the following three types.
\begin{enumerate}[I)]
    \item Arbitrary patient-fit scores $\pref(S)\in\mathbb{Q}$ or $\N_0$, or rather, since the patient set is finite, $\pref(S)\in\{0,1,\ldots,M\}$ for an $M\in\N$.
    An example for such arbitrary scores is the age difference of patients. \label{itm:general-score-def}
    \item Scores that indicate a good fit versus indifference $\pref(S)\in\{0,1\}$, e.g., indicating whether patients speak a common language.\label{itm:0-1-pref-matr}
    \item Scores that indicate a good fit versus indifference $\pref(S)\in\{0,1\}$ but also have a special structure, e.g., the patients can be divided in different categories, e.g., by age class, and  have a good fit, i.e., $\pref(S)=0$ iff all patients in $S$ are in the same category and $\pref(S)=1$ otherwise.\label{itm:class-score-def}
\end{enumerate}

Although the reports in literature are ambiguous how well age difference is suited to predict a good patient fit, we use it in our computational study since it is the only criterion that is included in our data.
We present four different functions to model age difference, at least one of type \ref{itm:general-score-def} to \ref{itm:class-score-def} each, and compare their impact on the runtime.
Additionally, we propose a scoring function that models the positive influence of mixing pre- and postoperative patients by assuming that a surgery takes place on the second day of each patient's stay.
Therefore, let $\age_p$ denote the age of patient $p\in\P$.

\begin{description}
    \item[Absolut age difference] We define $\pref(S)$ as the maximal age difference of patients in $S$, i.e.,
        \[
            \pref(S):=\max_{p\in S} \{\age_p\} - \min_{p\in S} \{\age_p\}.
        \]
        For $\abs{S}=2$, the value $\pref(S)$ corresponds to the absolut age difference of the two patients and, for $\abs{S}=1$, the value $\pref(S)=0$.
        This defines a patient-fit score of Type \ref{itm:general-score-def}.
        
    \item[Bounded age difference] We request that the age difference between two patients in $S$ is less or equal to $k\in\N$, i.e.,
        \[
            \pref(S):=\begin{cases}
                0, &\text{if } \left(\max_{p\in S} \{\age_p\}\right) - \left(\min_{p\in S} \{\age_p\}\right)\leq k\\
                1, &\text{otherwise.}
        \end{cases}\]
        This defines a patient-fit score of Type \ref{itm:0-1-pref-matr}. 
    \item[Age classes] We divide the patients into different age classes of size $k\in\N$ and either count the number of age classes represented in $S$, i.e.,
        \[
            \pref(S):=\abs{\left\{\ceil*{\frac{\age_p}{k}} \mid p\in S\right\}}
        \]
        or we check whether all patients in $S$ are in the same age class, i.e.,
        \[
            \pref(S):=\begin{cases}
                0, &\text{if } \abs{\left\{\ceil*{\frac{\age_p}{k}}\mid p\in S\right\}}=1,\\
                1, &\text{otherwise.}
        \end{cases}\]
    	This defines a patient-fit score of Type \ref{itm:general-score-def} or \ref{itm:0-1-pref-matr}, respectively.
	\item[Weighted age difference] We measure the relative age difference of two patients, with an offset $\epsilon>0$ to ensure a finite value, i.e.,
		\[
			\pref(S)_\epsilon \coloneqq \frac{\max_{p\in S}\{\age_p\}+\epsilon}{\min_{p\in S}\{\age_p\}+\epsilon}.
		\]
		Notably, $\pref(S)_\epsilon-1$ directly corresponds to the numeric percentage of deviation. This defines a patient-fit score of Type \ref{itm:general-score-def}.
		
\end{description}

Of course, there are various further possibilities to assess the similarity of roommates' age.
Especially for larger room sizes, the following two models can be interesting:
asking that every patient has at least one roommate of similar age, i.e.,
        \[
            \pref(S):=\begin{cases}
                0, &\text{if } \forall p\in S\exists q\in S\setminus\{p\}:\, \abs{\{\age_p\} - \{\age_p\}}\leq k \text{ (or } \ceil*{\frac{\age_p}{k}}=\ceil*{\frac{\age_p}{k}}\text{)},\\
                1, &\text{otherwise,}
        \end{cases}\]
        or
asking that the age classes represented in $S$ are balanced.
        Therefore, let $k\in\N$ be the age classes' size and $l:=\abs{\{ \ceil*{\frac{\age_p}{k}}\mid p\in S\}}$ the number of age classes in $S$.
        Further, let $S_1,\ldots,S_l\subseteq S$ be a partition of $S$ based on the patients' age classes, i.e.,
        $\bigcup_{i=1}^{l} S_i = S$ and $S_i\cap S_j=\emptyset$ for all $1\leq i < j\leq l$, with
        \[\bigabs{\left\{\ceil*{\frac{\age_p}{k}}\mid p\in S_i\right\}}=1\]
        for all $1\leq i \leq l$.
        Then, we define 
        \[
            \pref(S):=\begin{cases}
                0, &\text{if } \abs{\abs{S_i}-\abs{S_j}}\leq x\quad \forall 1\leq i<j\leq l,\\
                1, &\text{otherwise,}
        \end{cases}\]
        for a constant $x\in\N_0$. Both functions define a patient-fit score of type \ref{itm:0-1-pref-matr}.
However, all our instances contain rooms with capacities less or equal to $2$.
Therefore, we do not include them in our computational study.

For modelling the pairing of pre- and post-surgery patients
we simply assume that every patient has surgery on their first or second day in hospital
since we lack data on the patient's surgery date:
\begin{description}
    \item[Pre/post surgery] We request that there is at least one time period between the admission dates of the earliest and latest arriving patients in $S$, i.e.,
        \[
            \pref(S):=\begin{cases}
                0, &\text{if } \left(\max_{p\in S} \{\arr_p\}\right) - \left(\min_{p\in S} \{\arr_p\}\right)> 1\\
                1, &\text{otherwise.}
        \end{cases}\]
        This defines a patient-fit score of Type \ref{itm:0-1-pref-matr}. 
\end{description}

The presented concepts can easily be transferred to
other criteria for a good roommate fit, e.g., common television preferences.

\section{Finding fitting roommates using linear programs}
\label{sec:pref:lp}
In~\cite{brandt2024privates}, we proposed and compared several \ac{ip} formulations for \ac{pra} with the objectives of minimizing transfers $(\ftrans)$ and maximizing single-room requests $(\fpriv)$ in terms of their computational performance.
In this section, we use the best performing \acp{ip} from \cite{brandt2024privates} and integrate roommate compatibility into them.
For more information on the objectives $\ftrans$ and $\fpriv$, we refer to \cite{brandt2024privates}.

Since all our instances contain only rooms with at most two beds, we focus on this case only and use binary variables $y$ to indicate whether two patients share a room.
We present again both general \acp{ip} that allow arbitrary transfers in \cref{subsec:pref:ip:general}, and \acp{ip} which prohibit transfers by construction in \cref{sec:pref:ip:notransfer}.

\subsection{General integer programming formulation}
\label{subsec:pref:ip:general}
In this section, we integrate the maximization of single-room requests into the best performing \ac{ip} (H) from \cite{brandt2024privates}.
To encode whether two patients $p,q\in\P$ share a room in time period $t\in\T$,
we use three-dimensional binary variables $y_{pqt}$:
\begin{equation}
	y_{pqt} =\begin{cases}
		1,  \text{if patients } p \text{ and } q \text{ share a room} \text{ in time period } t,\\
		0,  \text{otherwise,}
	\end{cases}\\
\label{var:ypqt}
\end{equation}
Thus, the total score of patients who are assigned to the same rooms is given by
\begin{equation}
\label{eq:ypqt:prefobj}
    \fpref:=\sum_{t\in\T}\sum_{\substack{p,q \in \P(t) \\ p\neq q}}w_{pq}y_{pqt}.
\end{equation}
Additionally to variables $y_{pqt}$, we use the following variables:
\begin{align}
x_{prt}&=\begin{cases}
            1,  &\text{if patient }  p \text{ is assigned to room } r \text{ in time period } t,\\
            0,  &\text{otherwise,}
        \end{cases}\label{var:xprt}\\
\delta_{prt}&=\begin{cases}
            1,  &\text{if patient } p \text{ is transferred from room } r \text{ to another room}\\ &\text{after time period } t\\
            0,  &\text{otherwise,}
        \end{cases}\label{var:deltaprt}\\
    g_{rt} &=\begin{cases}
            1,  &\text{if there is a female patient assigned to room } r \text{ in time}\\ &\text{period } t,\\
            0,  &\text{otherwise.}
        \end{cases}\label{var:grt}\\
s_{prt} &=\begin{cases}
        1,  &\text{if } p \text{ is alone in room } r \text{ in time period } t,\\
        0,  &\text{otherwise.}
    \end{cases}\label{var:sprt}
\end{align}

As in \ac{ip} (H), we ensure that all patients are assigned to rooms for every time period of their stay using
\begin{equation}
\label{eq:everypatientAroom}
\sum_{r\in\R} x_{prt} = 1   \quad \forall t\in\T, p\in\P(t),
\end{equation}
and use the combined capacity, sex-separation, and single-room constraints
\begin{align}
    \sum_{p\in\fp(t)} x_{prt} + \sum_{p\in\fp\cap\priv(t)} (\rc_r -1)s_{prt} &\leq \rc_r g_{rt} &&\forall t\in\T,\ r\in\R,\label{eq:single:grt}\\
    \sum_{p\in\mp(t)} x_{prt} + \sum_{p\in\mp\cap\priv(t)} (\rc_r -1)s_{prt} &\leq \rc_r (1-g_{rt}) &&\forall t\in\T,\ r\in\R,\label{eq:single:grt:only}
\end{align}
with
\begin{equation}
    s_{prt} \leq x_{prt}   \quad\forall t\in\T,\ p\in\priv(t),\ r\in\R.\label{eq:single:s}
\end{equation}
We then model the choice of roommates via
\begin{equation}
    y_{pqt}\geq x_{prt}+x_{qrt}-1 \quad\forall t\in\T,\ p,q\in\P(t),\ p\neq q. \label{eq:pref}
\end{equation}

\citet{brandt2024privates} showed that it makes sense to fix $\fpriv$ to its optimum, if $\fpriv$ has the highest priority, instead of modelling it directly as objective function using
\begin{equation}
    \sum_{p\in\priv(t)}\sum_{r\in\R} s_{prt} \geq \smax_t \quad \forall t\in\T. \label{eq:fix:smax_t}
\end{equation}
Analogously, we can fix $\fpref$ to its minimum, alternatively to minimizing $\fpref$,
using
\begin{equation}
    \sum_{\substack{p,q \in \P(t) \\ p\neq q}}w_{pq}y_{pqt} \leq \wmin_t \quad \forall t\in\T. \label{eq:fix_wmin}
\end{equation}

In this section, we showcase how the objective setting and choice of scoring function influences the runtime.
We use the insights gained in our previous computational studies to reduce the number of setups to evaluate and compare the following three \ac{ip} formulations:
\begin{enumerate}[(A)]
\setcounter{enumi}{16}
    \item $\max~(-\ftrans,\fpriv,-\fpref)$ s.t. \cref{eq:everypatientAroom,eq:single:grt,eq:single:grt:only,eq:single:s,eq:pref}, \label{IP:Q}
    \item $\min~(\fpref,\ftrans)$ s.t. constraints of \ac{ip} \ref{IP:Q}, \cref{eq:fix:smax_t}, \label{IP:R}
    \item $\min~\ftrans$ s.t. constraints of \ac{ip} \ref{IP:R}, \cref{eq:fix_wmin}. \label{IP:S}
\end{enumerate}
The objectives' order determines again their priority in optimisation, i.e., $\min (\fpref,\ftrans)$ means that first $\fpref$ is minimised and then $\ftrans$.
\ac{ip} \ref{IP:Q} is the direct extension of \ac{ip} (H) where we
give $\fpriv$ a higher priority than $\fpref$ as this models the reality on most wards according to personal communication with practitioners.
With \ac{ip} \ref{IP:R}, we compute the optimal value for the patient compatibility under the condition that the maximum number of single-room requests is fulfilled while minimizing the number of needed transfers as second objective.
The resulting values are interesting to assess the sensibility of an integration of $\wmin$ into the \ac{ip} formulation, which we then evaluate using \ac{ip} \ref{IP:S}.
Additionally, we derive from a comparison of \acp{ip} \ref{IP:R} and \ref{IP:Q} whether it is computationally easier to optimize $\ftrans$ or $\fpref$.
However, our code can easily be adapted to evaluate other objective settings \cite{unserCode-v3}.

To perform the computational study, we further have to choose a tangible scoring function.
First computational results showed that the bounded-age-difference scoring function performs best among all scoring functions proposed in \cref{sec:pref:patient-fit} while the weighted-age-difference scoring function performs worst.
Therefore, we use those two scoring functions to evaluate the performance of \acp{ip} \ref{IP:Q}, \ref{IP:R}, and \ref{IP:S}.
Again, our code can easily be adapted to evaluate other scoring functions \cite{unserCode-v3}.

\begin{figure}[bt!]
    \centering
    \includegraphics[width=0.7\textwidth]{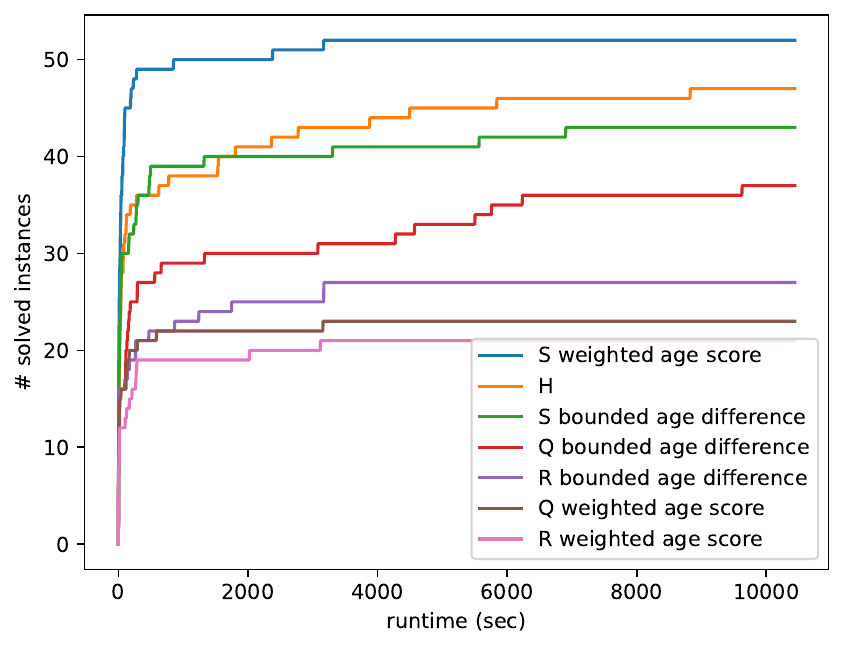}
    \caption{Comparison of IPs \ref{IP:Q} - \ref{IP:S} using 62 real-life instances, maximum runtime 4h}
    \label{fig:Q-S:runtime:14040}
\end{figure}
The formulations for \acp{ip} \ref{IP:Q} to \ref{IP:S}
were implemented in \emph{python} 3.10.4 and solved using \emph{Gurobi 11.0.0} \cite{gurobi} with a time limit of 4 hours.
All computations were done on the \href{https://www.itc.rwth-aachen.de/cms/IT-Center/Services/Servicekatalog/Hochleistungsrechnen/~hisv/RWTH-Compute-Cluster/?lidx=1}{RWTH High Performance Computing Cluster} using CLAIX-2018-MPI with Intel Xeon Platinum 8160 Processors “SkyLake” (2.1 GHz, 16 CPUs per task, 3.9 GB per CPU).
The code can be found at \cite{unserCode-v3}.
We illustrate in \cref{fig:Q-S:runtime:14040} on the y-axis the number of instances that have been solved to optimality or proven to be infeasible within the time frame depicted on the x-axis.
To assess the runtime-costs for integrating patient compatibility, we also include again the performance of \ac{ip} (H).

First, we notice that \ac{ip} \ref{IP:S}, depending on the scoring function, performs even better than \ac{ip} (H) or only slightly worse.
However, we have to keep in mind that \ac{ip} \ref{IP:S} is proven to be infeasible for $32$ ($19$) instances using the weighted-age-difference (bounded-age-difference) scoring function.
If we consider only instances for which \ac{ip} \ref{IP:S} is feasible, \ac{ip} (H) performs better than \ac{ip} \ref{IP:S}.
Further, we notice that \ac{ip} \ref{IP:Q} outperforms \ac{ip} \ref{IP:R} which indicates that optimizing $\fpref$ is computationally harder than optimizing $\ftrans$.
Moreover, if we consider again only instances where \ac{ip} \ref{IP:S} is feasible, \ac{ip} \ref{IP:Q} performs also better than \ac{ip} \ref{IP:S}.
Hence, the runtime advantage of \ac{ip} \ref{IP:S} lies in proving infeasibility, i.e., that the $\wmin$ cannot be achieved if single-room requests are fixed to maximum.
Overall, our results show that the choice of scoring function influences the \ac{ip}'s runtime more than the choice of \acp{ip}.

Regarding the achieved objective values, we observe that no transfers are needed in all optimal solutions found for \acp{ip} \ref{IP:Q} and \ref{IP:S} and in all but one optimal solutions for \ac{ip} \ref{IP:R}.
This result is impressive as all but three optimal solutions found for \ac{ip} \ref{IP:R} achieve $\fpref=\wmin$ and all but six optimal solutions found for \ac{ip} \ref{IP:Q}.
Therefore, we evaluate \ac{ip} formulations that prohibit transfers by construction in the next section.

\subsection{Integer programming formulation without transfers}
\label{sec:pref:ip:notransfer}
Analogously to \cite{brandt2024privates}, we propose and evaluate \ac{ip} formulations where transfers are prohibited by construction.
Therefore, we integrate in this section the choice of roommates into the best performing \ac{ip} (P) from \cite{brandt2024privates}.

Since our \ac{ip} forbids arbitrary transfers, we can restrict ourselves to two-dimensional binary variables $y_{pq}$ that encode whether two patients $p,q\in\P$ share a room at some time during their stay:
\begin{equation}
	y_{pq} =\begin{cases}
		1,  \text{if patients } p \text{ and } q \text{ share a room},\\
		0,  \text{otherwise.}
	\end{cases}\\
\label{var:ypq}
\end{equation}
To model the total roommate score correctly,
we calculate the length of the common stay of patients $p$ and $q$ as
\begin{equation}
    \los(p,q):= \min\{\dis_p,\dis_q\} - \max\{\arr_p,\arr_q\}\label{eq:los:pq}
\end{equation}
and the set of potential roommates $\potrmm$ as the tuples of patients with the same sex and intersecting hospital stays, i.e.,
\begin{equation}
    \potrmm:=\left\{(p,q) \mid (p,q\in \fp \text{ or } p,q\in \mp) \text{ and } \los(p,q) > 0\right\}.
\end{equation}
Then, the total roommate score is given by
\begin{equation}
\label{eq:ypq:prefobj}
    \fpref:=\sum_{\substack{(p,q) \in \potrmm}} \los(p,q) w_{pq} y_{pq}.
\end{equation}

Additionally to variables $y_{pq}$, we use the following binary varaibles $x_{pr}$, $g_{rt}$, and $s_{prt}$ for $p\in\P$, $r\in\R$, and $t\in\T$:
\begin{align}
x_{pr}&=\begin{cases}
            1,  &\text{if patient }  p \text{ is assigned to room } r,\\
            0,  &\text{otherwise,}
        \end{cases}\label{var:xpr}\\
    g_{rt} &=\begin{cases}
            1,  &\text{if there is a female patient assigned to room } r \text{ in time}\\ &\text{period } t,\\
            0,  &\text{otherwise.}
        \end{cases}\tag{\ref{var:grt}}\\
s_{prt} &=\begin{cases}
        1,  &\text{if } p \text{ is alone in room } r \text{ in time period } t,\\
        0,  &\text{otherwise.}
    \end{cases}\tag{\ref{var:sprt}}
\end{align}

As in \ac{ip} (P), we first ensure that all patients are assigned to rooms for every time period of their stay
\begin{equation}
\label{eq:xpr:everypatientAroom}
    \sum_{r\in\R} x_{pr} = 1   \quad \forall p\in\P,
\end{equation}
fix the previous patient-room assignments via
\begin{equation}\label{eq:xpr:prefix}
    x_{pr} = 1 \quad \forall (p,r)\in\rpold,
\end{equation}
and use the combined capacity, sex-separation, and single-room constraints
\begin{align}
    \sum_{p\in\fp(t)} x_{pr} + \sum_{p\in\fp\cap\priv(t)} (\rc_r -1)s_{prt} &\leq \rc_r g_{rt} &&\forall t\in\T,\ r\in\R, \label{eq:xpr:single:grt}\\
    \sum_{p\in\mp(t)} x_{pr} + \sum_{p\in\mp\cap\priv(t)} (\rc_r -1)s_{prt} &\leq \rc_r (1-g_{rt}) &&\forall t\in\T,\ r\in\R, \label{eq:xpr:single:grt:only}
\end{align}
with
\begin{equation}
    s_{prt} \leq x_{pr}   \quad\forall t\in\T,\ p\in\priv(t),\ r\in\R.\label{eq:xpr:single:s}
\end{equation}
We also again fix $\fpriv$ to its optimal value via
\begin{equation}
    \sum_{p\in\priv(t)}\sum_{r\in\R} s_{prt} \geq \smax_t \quad \forall t\in\T, \tag{\ref{eq:fix:smax_t}}
\end{equation}
and then model the choice of roommates via
\begin{equation}
    y_{pq}\geq x_{pr}+x_{qr}-1 \quad\forall t\in\T,\ p,q\in\P(t),\ p\neq q. \label{eq:pr:pref}
\end{equation}

Analogously to \cref{sec:pref:ip:notransfer}, we can fix $\fpref$ to its minimum, alternatively to minimizing $\fpref$,
using
\begin{equation}
    \sum_{\substack{p,q \in \fp(t) \\ \los(p,q)>0}}w_{pq}y_{pq} + \sum_{\substack{p,q \in \mp(t) \\ \los(p,q)>0}}w_{pq}y_{pq} \leq \wmin_t \quad \forall t\in\T. \label{eq:fix_wmin:pr}
\end{equation}

Analogously to \cref{subsec:pref:ip:general}, we compare the following \ac{ip}-formulations
\begin{enumerate}[(A)]
\setcounter{enumi}{19}
    \item $\max~(\fpriv,-\fpref)$ s.t. \cref{eq:xpr:everypatientAroom,eq:xpr:prefix,eq:xpr:single:grt,eq:xpr:single:grt:only,eq:xpr:single:s,eq:pr:pref}, \label{IP:T}
    \item $\min~\fpref$ s.t. constraints of \ac{ip} \ref{IP:T}, \cref{eq:fix:smax_t}, \label{IP:U}
    \item $\min~0$ s.t. constraints of \ac{ip} \ref{IP:U}, \cref{eq:fix_wmin:pr},  \label{IP:Va}
\end{enumerate}
and use the bounded-age-difference, and  the weighted-age-difference scoring functions to evaluate their performance.
However, our code can easily be adapted to evaluate other objective settings and scoring functions \cite{unserCode-v3}.

\begin{figure}[bt!]
    \centering
    \includegraphics[width=0.7\textwidth]{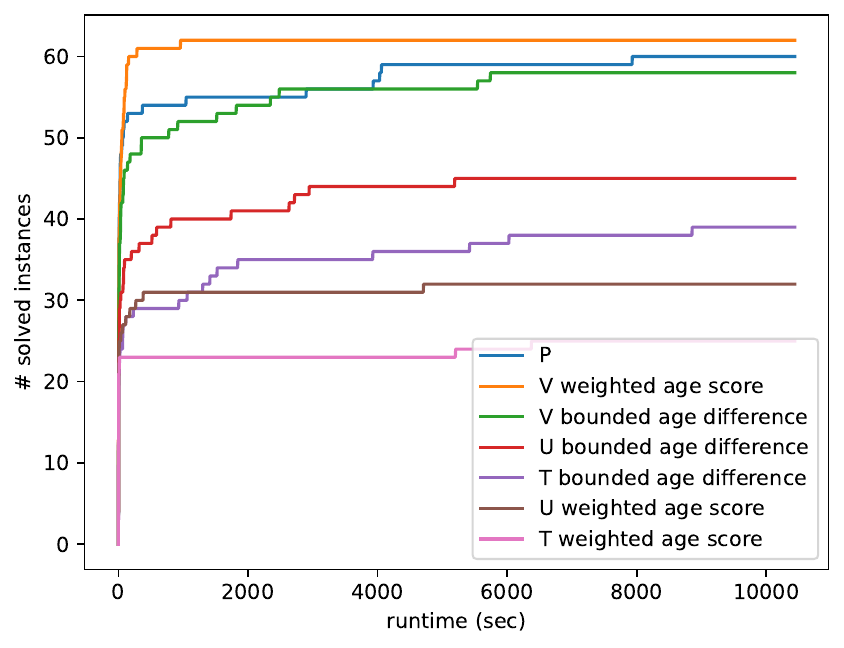}
    \caption{Comparison of IPs \ref{IP:T} - \ref{IP:Va} using 62 real-life instances, maximum runtime 4h}
    \label{fig:T-V:runtime:14040}
\end{figure}
We observe that \ac{ip} \ref{IP:Va} performs comparable to \ac{ip} (P).
However, we have to keep in mind that \ac{ip} \ref{IP:Va} is proven to be infeasible for $31$ ($42$) instances for the bounded-age-difference (weighted-age-difference) scoring function while for \ac{ip} (P) only $17$ instances are infeasible.
If we consider only instances that are feasible for \ac{ip} \ref{IP:Va}, we observe that \ac{ip} (P) performs best, followed by \ref{IP:Va} which performs only slightly better than \ref{IP:U}.
Moreover, our results show again the siginificant influence of the chosen scoring function on the runtime of all proposed \ac{ip} formulations.
We will examine this influence further in the next section, where we integrate the concept of roommate compatibility into our dynamic solution approach and evaluate all five scoring functions that we proposed in \cref{sec:pref:patient-fit}.

\section{Integrating the roommate fit into dynamic PRA}
\label{sec:pref:dyn}
In this section, we use the developed \ac{ip} formulations from \cref{subsec:pref:ip:general,sec:pref:ip:notransfer} to integrate roommate compatibility into solution approach
that we proposed in \cite{brandt2024privates}
for the \acf{dpra}.
In \ac{dpra}, the patient stays become known iteratively at their registration date.
The task is then to find in every time period a patient-to-room assignment for all registered patients that respects the previous room assignment of already arrived patients.
We denote the set of those pre-assigned room assignments with \[\F\subset \{p\in \P \mid \arr_p=0\}\times \R.\]

Naturally, we would use \ac{ip} \ref{IP:Va} as first \ac{ip} in our solution approach as it performed best on the large instances in the previous section.
However, we have to weigh the advantage regarding the optimization time of \ac{ip} \ref{IP:Va} against the multiple computation of $\wmin_t$ which has to be done in each iteration for every remaining time period of the planning horizon.
Therefore, it might be more efficient to use \ac{ip} \ref{IP:U} instead of \ref{IP:Va} in the initial step of our algorithm, although \ac{ip} \ref{IP:Va} performed better on the large instances in the previous section.
Furthermore, we compare the impact of all five scoring functions that we proposed in \cref{sec:pref:patient-fit} on the algorithm's runtime.

In general, the idea of our approach remains the same as in \cite{brandt2024privates}, since we still optimize primarily for $\fpriv$ while keeping the runtime acceptable for practical purposes (based on conversations with practitioners):
First, we check combinatorially whether the instance is feasible.
Second, we check whether we can find a solution that is optimal for $\fpriv$, $\ftrans$, and $\fpref$ either using \ac{ip} \ref{IP:Va} or \ac{ip} \ref{IP:U}.
If not, we third ask for an optimal solution w.r.t. $\fpriv$ and maximize $\fpref$ while transferring only currently present patients (\emph{same-day transfers}). 
If this is still not possible, we relax in the fourth step the fixation of $\fpriv$ to $\smax$ and maximize first the number of fulfilled single-room requests while only using same-day transfers. 
If we still find no feasible solution, we finally allow arbitrary transfers and use \ac{ip} \ref{IP:Q} to compute a feasible solution.
After successful computation, we fix all patient-room assignments for patients that are in hospital in the current time period by adding them to set $\rpold$ while removing outdated ones.
We then update the patient set and continue analogously with the next time period.
\begin{figure}[tb!]
    \centering
    \scalebox{0.5}{
        \begin{tikzpicture}[node distance=2cm,every label/.style={align=left}]
    
    \tikzstyle{startstop} = [rectangle, rounded corners, minimum width=2.5cm, minimum height=1cm,text centered, draw=black, fill=rwth-red!50]

    \tikzstyle{process} = [rectangle, minimum width=1.5cm, minimum height=1cm, text centered, draw=black, fill=combi-orange]
    
    \tikzstyle{decision} = [diamond, minimum width=2cm, minimum height=2cm, text centered, draw=black, fill=combi-green]
    
    \tikzstyle{YN} = [minimum width=0cm, minimum height=0cm, text centered]

    \tikzstyle{arrow} = [thick,->,>=stealth]
        
    
    \node (feasibility) [decision] {Feasible?};
    \node (initialisation) [startstop, below of=feasibility] {Initialisation};
    \node (no0) [YN,below of=feasibility, yshift=1.6cm,xshift=1.4cm] {No};
    \node (yes0d) [YN,below of=feasibility, yshift=1cm,xshift=0.6cm] {Yes};
     \draw [arrow] (feasibility) -- (initialisation);
    
    \node (N) [process, below of=initialisation] {\Large V/U}; 
    \node (N_eval) [decision, right of=N, xshift=1cm] {Feasible?};
    \node (Inc) [decision, below of=N_eval,yshift=-1cm] {\large$t=\nT?$};
    
    \draw [arrow] (initialisation) -- (N);
    \draw [arrow] (N) -- (N_eval);
    \draw [arrow] (N_eval) -- (Inc);
    
    \node (no1) [YN,below of=N_eval, yshift=1.6cm,xshift=1.4cm] {No};
    \node (yes1) [YN,below of=N_eval, yshift=0.5cm,xshift=0.4cm] {Yes};
    
    \node (no3) [YN,below of=Inc, yshift=1.6cm,xshift=-1.4cm] {No};
    \node (yes3) [YN,below of=Inc, yshift=0.5cm,xshift=0.4cm] {Yes};
    
    \node (C) [process, right of=N_eval, xshift=1cm] {\Large U*}; 
    \node (C_eval) [decision, right of=C, xshift=1cm] {Feasible?};
    
    \draw [arrow] (N_eval) -- (C);
    \draw [arrow] (C) -- (C_eval);
    \draw [arrow] (C_eval) |- (Inc);
    
    \node (no2) [YN,below of=C_eval, yshift=1.6cm,xshift=1.4cm] {No};
    \node (yes2) [YN,below of=C_eval, yshift=0.5cm,xshift=0.4cm] {Yes};
    
    \node (O) [process, right of=C_eval, xshift=1cm] {\Large T*}; 
    \node (O_eval) [decision, right of=O, xshift=1cm] {Feasible?};
    
    \draw [arrow] (C_eval) -- (O);
    \draw [arrow] (O) -- (O_eval);
    \draw [arrow] (O_eval) |- (Inc);
    
    \node (noO2) [YN,below of=O_eval, yshift=1.6cm,xshift=1.4cm] {No};
    \node (yesO2) [YN,below of=O_eval, yshift=0.5cm,xshift=0.4cm] {Yes};
    
    \node (E) [process, right of=O_eval, xshift=1cm] {\Large Q }; 
    \node (E_eval) [decision, right of=E, xshift=1cm] {Feasible?};
    \node (INF) [startstop, right of=E_eval, yshift=-4.8cm] {Terminate};
    
    \node (no4) [YN,below of=E_eval, yshift=1.6cm,xshift=1.4cm] {No};
    \node (yes4) [YN,below of=E_eval, yshift=0.5cm,xshift=0.4cm] {Yes};
    
    \draw [arrow] (O_eval) -- (E);
    \draw [arrow] (E) -- (E_eval);
    \draw [arrow] (E_eval) |- (Inc);
    \draw [arrow] (E_eval) -| (INF);
    \draw [arrow] (Inc) |- (INF);
    \draw [arrow] (feasibility) -| (INF);
    
    \draw [arrow] (Inc) -| (N);

    \node (IncLabel1) [YN,below of=Inc, yshift=1.6cm,xshift=-3.2cm] {Update $\P,\rpold$};
    \node (IncLabel2) [YN,below of=Inc, yshift=1.2cm,xshift=-3.2cm] {$t \pluseq 1$};

\end{tikzpicture}
    }
    \caption{Algorithm for dynamic PRA optimization of single-room requests, patient compatibility and transfers}
    \label{fig:modelcombi:dyn:pref}
\end{figure}
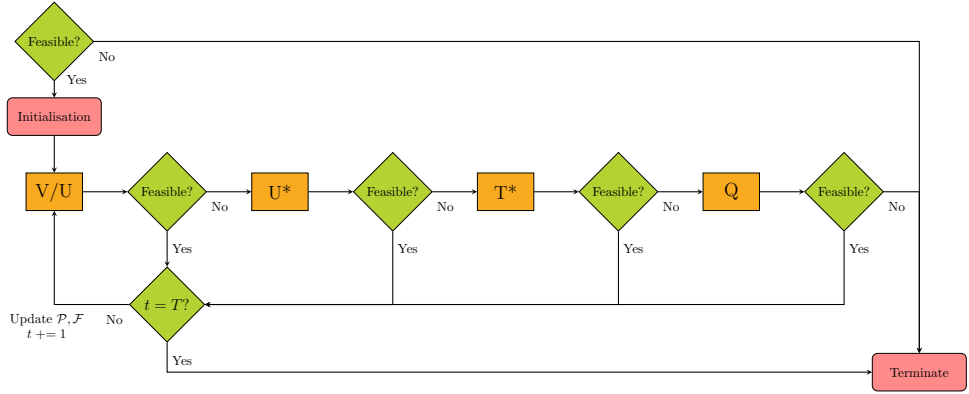

To formalize this solution apporoach, we need variants of \acp{ip} \ref{IP:T} and \ref{IP:U} that allow same-day transfers:
\begin{enumerate}[(A*)]
\setcounter{enumi}{19}
    \item $\max~(\fpriv,-\fpref,\sum_{(r,p)\in\rpold} x_{pr}$) s.t. constraints of \ref{IP:T} without \cref{eq:xpr:prefix}\label{IP:Tstar}
    \item $\max~(-\fpref,\sum_{(r,p)\in\rpold} x_{pr}$) s.t. constraints of \ref{IP:U} without \cref{eq:xpr:prefix}\label{IP:Ustar}
\end{enumerate}
A visualization of our solution approach is provided in \cref{fig:modelcombi:dyn:pref}.

We evaluate our algorithm again on our $62$ real-world instances spanning a whole year.
As a result, we get that all instances use $365$ iterations of the algorithm and all are solved within less than $600$ seconds per year, cf.~\cref{fig:dyn:runtime:pref}.
\begin{figure}[bt!]
\begin{subfigure}[b]{0.5\textwidth}
    \centering
    \includegraphics[width=\textwidth]{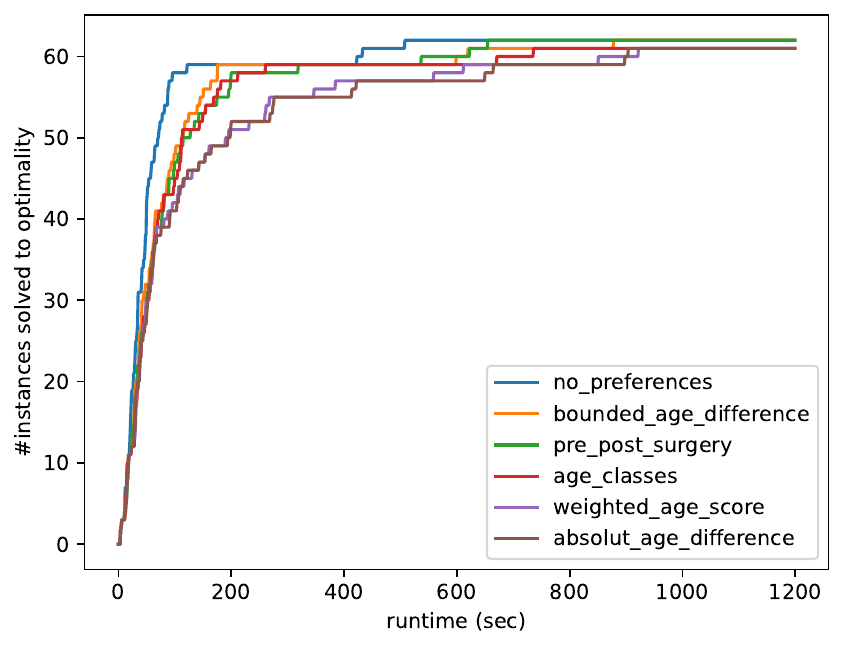}
    \caption{\ac{ip} \ref{IP:U} as first \ac{ip}}
    \label{fig:dyn:runtime:pref:U}
\end{subfigure}
\begin{subfigure}[b]{0.5\textwidth}
    \centering
    \includegraphics[width=\textwidth]{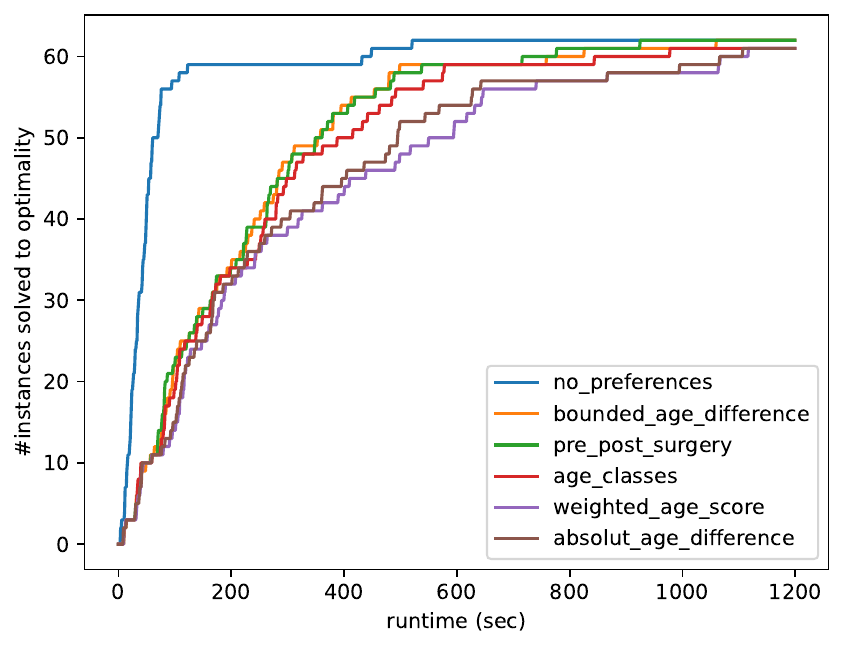}
    \caption{\ac{ip} \ref{IP:Va} as first \ac{ip}}
    \label{fig:dyn:runtime:pref:V}
\end{subfigure}
    \caption{Runtime of algorithm for dynamic PRA with $\nT=365$}
    \label{fig:dyn:runtime:pref}
\end{figure}
For application purposes however, the runtime per iteration is more interesting than the total runtime of $365$ iterations.
Therefore, we report in \cref{fig:dyn:results:perIter:pref} the runtime of all $2\cdot 6\cdot 62\cdot 365=271560$ iterations on a log-scale axis using boxplots.
We aggregate the results for the different scoring functions into one boxplot as separate (log-scaled) boxplots are identical to the eye.
Hence, the observed different runtimes are the result of statistical outliers, i.e., a minority of iterations that take significantly longer to solve.

The results show that more than $95$\% of the iterations are solved well in less than a second (for any scoring function) if we use \ac{ip} \ref{IP:U} in the first step,  cf.~\cref{fig:dyn:runtime:box:pref}.
We further observe that, in this case, the use of a scoring function increases the range of runtimes while decreasing the median runtime in comparison to ignoring patient compatibility.
If we use \ac{ip} \ref{IP:Va} in the first step, we observe again that the use of a scoring function increases the range of runtimes but also increases significantly the median runtime. 
In this case, only about $90$\% of the iterations are solved in less than a second (for any scoring function).

\begin{figure}[tb!]
\begin{subfigure}[b]{0.5\textwidth}
    \centering
    \includegraphics[width=\textwidth]{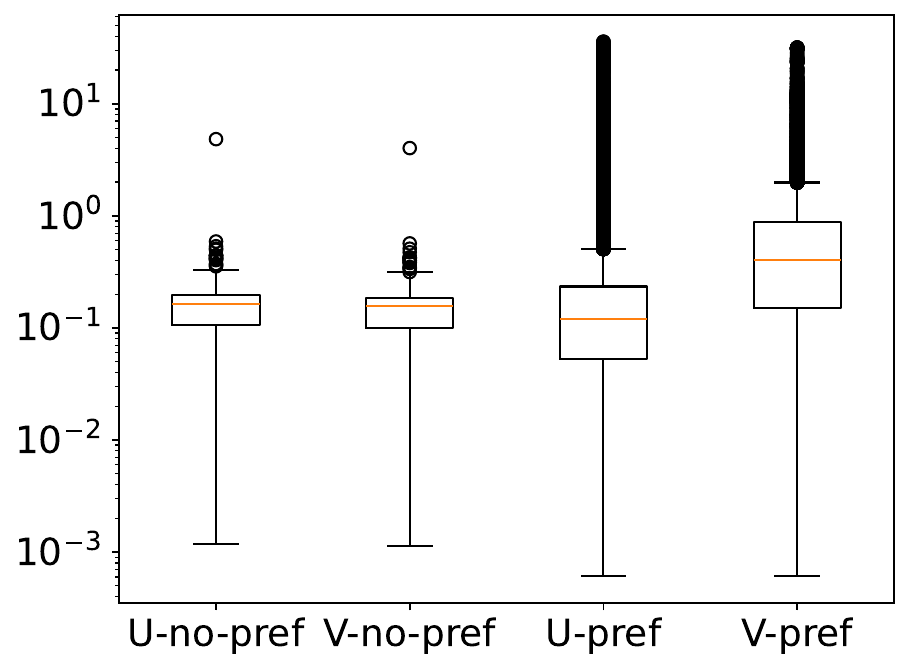}
    \caption{total runtime}
    \label{fig:dyn:runtime:box:pref}
\end{subfigure}
\begin{subfigure}[b]{0.5\textwidth}
    \centering
    \includegraphics[width=\textwidth]{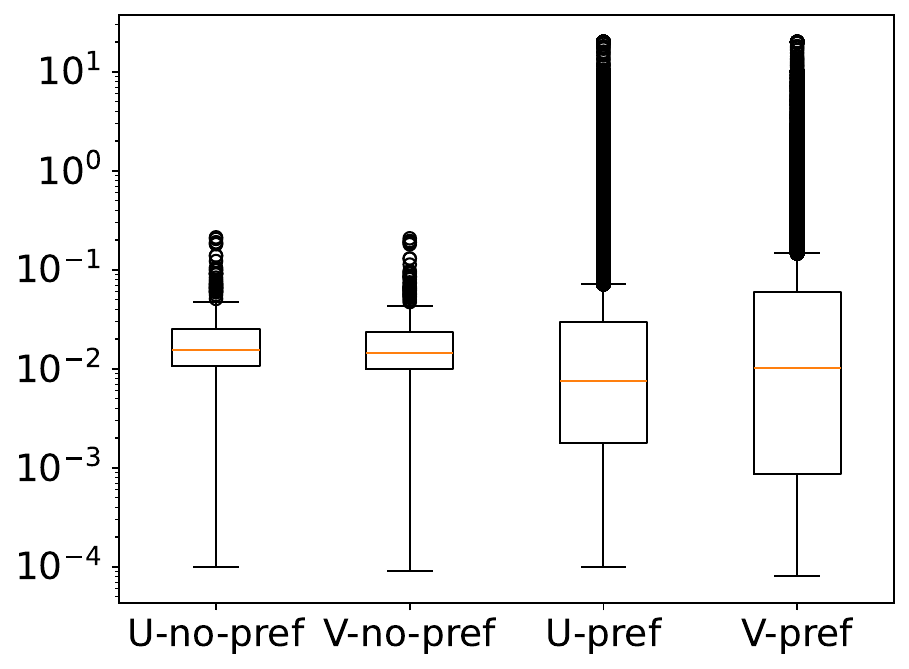}
    \caption{time used for solving \acp{ip}}
    \label{fig:dyn:opttime:box:pref}
\end{subfigure}
\caption{Runtime per iteration of the algorithm for dynamic PRA}
\label{fig:dyn:results:perIter:pref}
\end{figure}

To evaluate the impact of the time needed for computing $\wmin$ in every iteration, we depict in \cref{fig:dyn:opttime:box:pref} for every iteration the time that was needed to solve all \acp{ip} in that iteration.
We aggregate again the results for the different scoring functions into one boxplot as separate (log-scaled) boxplots are identical to the eye.
We observe that the difference between both versions is much smaller if we consider only the needed optimization time.
Hence, the computation of $\wmin$ using a minimum weighted perfect matching problem is the reason for the slow runtime when using \ac{ip} \ref{IP:Va}.
Our results show further that, if no scoring function is used, it is slightly faster to use \ac{ip} \ref{IP:Va} in the first step.
However, if roommate compatibility is optimized, then it is overall faster to use \ac{ip} \ref{IP:U} in the first step.
Remark that the results also show that there is a significant amount of instances where the using \ac{ip} \ref{IP:Va} in the first step is faster.
Therefore, we also evaluated whether using a rough estimate for $\wmin$ as bound leads to a computational advantage of \ac{ip} \ref{IP:Va} by using $2\cdot\wmin$ instead of $\wmin$.
The boxplots of the corresponding runtimes and optimization times per iteration, however, showed no noticeable differences to those depicted in \cref{fig:dyn:results:perIter:pref}.
Hence, the time needed to compute $\wmin$ in every iteration significantly impacts the iteration's runtime but it is not the decisive factor regarding the decision of using \ac{ip} \ref{IP:U} or \ref{IP:Va} in the first step.

In conclusion, we use the bound $\wmin$ to assess again the quality of the computed solutions:
Using \ac{ip} \ref{IP:U} or \ref{IP:Va} has no significant influence on the solution quality.
In both cases, the optimal overall compatibility score was achieved for the same $19$-$26$ instances, depending on the scoring function.
This result is especially impressive as the required amount of transfers and fulfilled single-room requests is comparable to the case where no compatibility of roommates is considered.
We further observe that for each chosen scoring function the same instances lead to high or low quality solutions with regard to roommate compatibility.
Determining the reasons behind this behaviour, however, is still an open question.

\section{Conclusion}\label{sec:conclusion}
In this paper, we presented novel combinatorial insights for the patient-to-room assignment problem with regard to choosing most compatible roommates.
We showed how an optimal partition of patients into roommates can be computed for wards with single and double rooms using known algorithms for the minimum weighted perfect matching problem.
A precondition is however that we can measure the compatibility of patients as roommates.
We gave an overview about criteria for high compatibility mentioned in literature, classified them, and proposed five specific scoring functions that we used in our computational studies.
The roommate compatibility of patients also links in with the ongoing scientific debate about the advantages and disadvantages of single vs. multiple bed room ward designs (e.g., see \cite{vandeGlind_2007,Persson_2012,Persson_2015}), which is not the focus of this work.
However, we would like to remark that many factors of considered advantages or disadvantages of multi-bed rooms, especially concerning social aspects, are highly dependent on the specific roommate chosen.

We further explored the performance of different \ac{ip} formulations for modelling patient compatibility.
Using all our insights, we integrated roommate compatibility into our \ac{ip}-based solution approach for \ac{dpra} which obtains high quality solutions in reasonable time.
One of our key insights here is that integrating our combinatorial insights regarding optimal compatibility of roommates is helpful to assess a solution's quality, but not for efficient optimization.
Another key insight is that $\fpref$ seems to take longer to optimize using \acp{ip} than $\fpriv$.
In general, the comparison of runtime and needed time for solving \acp{ip} showed that solving the \acp{ip} takes only a fraction of the overall runtime of an iteration.
Hence, the runtime can very likely be improved by choosing a more efficient implementation of the algorithm.
Our current implementation is designed to allow easy evaluation of different variants.
Optimizing the implementation with regard to the runtime is planned for future work as well as an evaluation of the performance of different solvers.

Overall, this paper provides a proof of concept that it is possible to efficiently integrate roommate compatibility into algorithms for \ac{pra} so that the runtime is suitable for real-world application, even if considering multiple objectives at the same time.

\section{Acknowledgements}
We thank the team at RWTH Aachen University hospital for their support.
We thank Jens Brandt for his input regarding efficient implementation and coding.
We further thank Daniel Müller for the extraordinary first take on patient compatibility in his master thesis. Although none of his results made it into the paper, it provided a good starting point for the research that lead to this paper.
Simulations were performed with computing resources granted by RWTH Aachen University.

\printbibliography
\end{document}